\documentclass[12pt]{amsart}

\usepackage{amssymb}
\usepackage{color}

\input{prepictex.tex}
 \input{pictex.tex}
 \input{postpictex.tex}

\theoremstyle{plain}
 \newtheorem{thm}{Theorem}[section]
 \newtheorem{cor}[thm]{Corollary}
 \newtheorem{lem}[thm]{Lemma}
 \newtheorem{prop}[thm]{Proposition}
 
\newtheorem{remark}[thm]{Remark}
\theoremstyle{definition}
 \newtheorem{ex}[thm]{Example}
\theoremstyle{remark}

\def\C{{\mathbb{C}}}
\def\T{{\mathbb{T}}}
\def\N{{\mathbb{N}}}
\def\Zb{{\mathbb{Z}}}
\def\Mb{{\mathbb{M}}}
\def\R{{\mathbb{R}}}

\def\D{{\mathcal {D}}}
\def\F{{\mathcal{F}}}
\def\OO{{\mathcal{O}}}
\def\U{{\mathcal{U}}}
\def\Z{{\mathcal{Z}}}

\def\Ad{\operatorname{Ad}}
\def\id{\operatorname{id}}
\def\span{\operatorname{span}}

\def\rc{\lambda_u(\F_n)'\cap\OO_n}
\def\rcc{\lambda_u(\F_2)'\cap\OO_2}
\def\frc{\lambda_u(\F_n)'\cap\F_n}

\begin{document}

\title[Endomorphisms of $O_n$]
{On Endomorphisms of the Cuntz Algebra which Preserve the Canonical UHF-Subalgebra, II}

\author[Tomohiro Hayashi]{{Tomohiro Hayashi} }
\address
{Nagoya Institute of Technology, 
Gokiso-cho, Showa-ku, Nagoya, Aichi, 466--8555, Japan}
\email{hayashi.tomohiro@nitech.ac.jp}

\author[Jeong Hee Hong]{{Jeong Hee Hong}}
\address{Department of Data Information, 
Korea Maritime and Ocean University, 
Busan 606--791, South Korea}
\email{hongjh@hhu.ac.kr}

\author[Wojciech Szyma\'{n}ski]{Wojciech Szyma\'{n}ski}
\address{Department of Mathematics and Computer Science, 
The University of Southern Denmark, 
Campusvej 55, DK--5230 Odense M, Denmark}
\email{szymanski@imada.sdu.dk}
\thanks{T. Hayashi was supported by JSPS KAKENHI Grant No. 25400109. 
J. H. Hong was supported by Basic Science Research Program through 
the National Research Foundation of Korea (NRF) funded by the Ministry of Education, Science and 
Technology (Grant No. 2012R1A1A2039991).
W. Szyma\'{n}ski was partially supported by the FNU Research Project `Operator algebras, 
dynamical systems and quantum information theory' (2013--2015), the Villum Fonden Research 
Grant `Local and global structures of groups and their algebras' (2014--2018), and by the Mittag-Leffler 
Institute during his stay there in January-February, 2016.}

\date{March 30, 2016}

\baselineskip=17pt

\maketitle

\begin{abstract}
It was shown recently by Conti, R{\o}rdam and Szyma\'{n}ski that there exist endomorphisms 
$\lambda_u$ of the Cuntz algebra $\OO_n$ such that $\lambda_u(\F_n)\subseteq\F_n$ but $u\not\in\F_n$, and a question was raised if for such a $u$ there must always exist a unitary $ v\in\F_n$ with $\lambda_u|_{\F_n} 
= \lambda_v|_{\F_n}$. In the present paper, we answer this question to the negative. To this end, 
we analyze the structure of such endomorphisms $\lambda_u$ for which the relative commutant 
$\lambda_u(\F_n)'\cap\F_n$ is finite dimensional. 
\end{abstract}


\section{Introduction and preliminaries}

This paper is devoted to continuation of the line of investigation of exotic endomorphisms of the Cuntz algebras 
initiated in \cite{CRS}. Our main result is solution of a question raised therein, see below for details. Our startegy 
is based on a detailed analysis of such endomorphisms $\lambda_u$ of $\OO_n$ that globally preserve the core 
UHF subagebra $\F_n$ and have finite dimensional relative commutant $\lambda_u(\F_n)'\cap\OO_n$, and 
builds on the earlier results in this direction obtained in \cite{H}. 

The Cuntz algebra $\mathcal{O}_{n}$, $n\geq 2$, is the $C^{*}$-algebra generated by isometries 
$S_{1},\dots,S_{n}$ satisfying $\sum_{i=1}^{n}S_{i}{S_{i}}^{*}=1$. It is a purely infinite, 
simple $C^*$-algebra, independent of the choice of generating isometries, \cite{Cun1}. 
We denote by $W_n^k$ the set of $k$-tuples $\mu = (\mu_1,\ldots,\mu_k)$
with $\mu_m \in \{1,\ldots,n\}$, and by $W_n$ the union $\cup_{k=0}^\infty W_n^k$,
where $W_n^0 = \{0\}$. If $\mu \in W_n^k$ then $|\mu| = k$ is the length of $\mu$. 
If $\mu = (\mu_1,\ldots,\mu_k) \in W_n$ then $S_\mu = S_{\mu_1} \ldots S_{\mu_k}$
($S_0 = 1$ by convention) is an isometry in $\OO_n$. 
Every word in $\{S_i, S_i^* \ | \ i = 1,\ldots,n\}$ can be uniquely expressed as
$S_\mu S_\nu^*$, for $\mu, \nu \in W_n$ \cite[Lemma 1.3]{Cun1}.

The gauge action $\gamma$ of the circle group ${\Bbb T}$ on $\mathcal{O}_{n}$ is defined 
by $\gamma_{z}(S_{i})=zS_{i}$, $z\in\T$. Let $\mathcal{F}_{n}$ be the fixed point algebra of $\gamma$. 
Denote $\F_n^{(k)}:=\span\{S_\mu S_\nu^* \mid \mu, \nu \in W_n^k\}$. Then  $\F_n$ is  
generated by $\F_n^{(k)}$, $k=1,2,\ldots$, and each $\F_n^{(k)}$ is isomorphic to the matrix 
algebra $M_{n^k}(\C)$. Thus $\F_n$ is isomorphic to the UHF-algebra of type $n^{\infty}$, and hence  it 
 has a unique tracial state 
$\tau$. There exists a faithful conditional expectation $E:\mathcal{O}_{n}\rightarrow \mathcal{F}_{n}$, 
defined by integration with respect to the Haar measure on $\T$ as 
$$
E(x)=\int_{\Bbb T}\gamma_{z}(x)dz.
$$
For each $k\in\Zb$ we denote by $\OO_n^{(k)}$ the corresponding spectral subspace for 
$\gamma$ in $\OO_n$, 
$$ \OO_n^{(k)} := \{x\in\OO_n \mid \gamma_z(x)=z^k, \;\forall z\in\T\}. $$
Thus, in particular, $\OO_n^{(0)}=\F_n$. 

The $C^*$-subalgebra of $\OO_n$ generated by projections $P_\mu:=S_\mu S_\mu^*$, $\mu\in W_n$, is a 
MASA (maximal abelian subalgebra) in $\OO_n$. We call it the diagonal and denote $\D_n$, 
also writing $\D_n^k$ for $\D_n\cap\F_n^{(k)}$. 

The canonical shift endomorphism $\varphi:\OO_n\to\OO_n$ is defined by 
$$ 
\varphi(x)=\sum_{i=1}^{n}S_{i}xS_{i}^{*}.
$$ 
It is easy to see that $S_{i}x=\varphi(x)S_{i}$ and $x{S_{i}}^{*}={S_{i}}^{*}\varphi(x)$ 
for all $x\in\OO_n$. 

As shown by Cuntz in \cite{Cun2}, there exists a bijective correspondence
between unitaries in $\OO_n$ (whose collection is denoted $\U(\OO_n)$) and unital 
$*$-endomorphisms of $\OO_n$, determined by
$$ 
\lambda_u(S_i) = u S_i, \;\;\; i=1,\ldots, n. 
$$
We have $\Ad(u)=\lambda_{u\varphi(u^*)}$ for all $u\in\U(\OO_n)$. 
If $u\in\U(\OO_n)$ then for each positive integer $k$ we denote
\begin{equation}\label{uk}
u_k = u \varphi(u) \cdots \varphi^{k-1}(u).
\end{equation}
Here $\varphi^0=\id$, and we agree that $u_k^*$ stands for $(u_k)^*$. If
$\alpha$ and $\beta$ are multi-indices of length $k$ and $m$, respectively, then
$\lambda_u(S_\alpha S_\beta^*)=u_kS_\alpha S_\beta^*u_m^*$. 

\medskip
The Cuntz correspondence between unitaries and endomorphisms of $\OO_n$ provides a very 
efficient tool for investigations of the latter. In this note, we continue the study (by several authors) 
of those unital endomorphisms which globally preserve the UHF-subalgebra $\F_n$. For example, such 
endomorphisms were analyzed from the point of view of the Jones-Kosaki-Watatani index theory 
in \cite{I} and \cite{CP}, and in connection with Hopf algebra actions in \cite{Cun3} and \cite{L}. 
More recently, interesting combinatorial approaches to the study of permutative endomorphisms of this 
type have been found (e.g. see \cite{CSTrans}, \cite{CHSAdvMath}, and a survey article \cite{CHSBanach}). 

It was observed by Cuntz in his groundbreaking paper \cite{Cun2} that an {\em automorphism} $\lambda_u$ 
globally preserves $\F_n$ if and only if $u\in\F_n$. The situation is more complex with 
{\em proper endomorphisms}. 
Clearly, $u\in\F_n$ implies $\lambda_u(\F_n)\subseteq\F_n$, \cite{Cun2}, but the question if the converse 
is true remained open until very recently. Indeed, it was shown in \cite{CRS} that there exist unitaries $u$ 
in $\OO_n\setminus\F_n$ such that $\lambda_u(\F_n)\subseteq\F_n$. All such examples found therein were of 
the form $u=wv$ with $w\in\lambda_{u}(\mathcal{F}_{n})'\cap\mathcal{O}_{n}$ and $v\in \mathcal{F}_{n}$. 
In such a case, we also have $\lambda_u(x)=\lambda_v(x)$ for all $x\in\F_n$. 
Thus a natural question arises if such a factorization of $u$ is always possible whenever $\lambda_u(\F_n)\subseteq\F_n$ (cf. \cite[Problem 5.3]{CHSBanach}). 

Some progress towards answering this question has been made recently in \cite{H} and \cite{HS}. 
The main purpose of the present paper is to develop definite methods for analyzing endomorphisms 
$\lambda_u$ of $\OO_n$ satifying $\lambda_u(\F_n)\subseteq\F_n$ and an additional condition that 
the relative commutant $\lambda_u(\F_n)'\cap\F_n$ be finite dimensional. In particular, we give 
a verifiable criterion for determining if the aforementioned decomposition is possible, 
Corollary \ref{maincriterion}. Based on this criterion, 
in Section 3 we give an explicit example of a unitary $u\in\OO_2$ such that $\lambda_u(\F_2)\subseteq
\F_2$ and $\dim \lambda_u(\F_2)'\cap\F_2 < \infty$ but there is no unitary $v\in\F_2$ such that 
$\lambda_u|_{\F_2} = \lambda_v|_{\F_2}$, see Example \ref{example1}. In this way, we answer 
to the negative the question raised in \cite{CRS} and \cite{CHSBanach}. 


\section{The relative commutants}
We begin by recording for future references a few simple facts, essentially contained in \cite{CRS} 
and \cite{H}. 

\begin{prop}\label{3conditions}
Let $u\in\U(\OO_n)$. Then the following conditions are equivalent.
\begin{enumerate}
\item[(1)]  $\lambda_{u}(\mathcal{F}_{n})\subseteq \mathcal{F}_{n}$, 
\item[(2)]  $\lambda_{\gamma_{z}(u)}|_{\mathcal{F}_{n}}=\lambda_{u}|_{\mathcal{F}_{n}}$ 
for all $z\in {\Bbb T}$, 
\item[(3)]  $u\gamma_{z}(u^{*})\in \lambda_{u}(\mathcal{F}_{n})'\cap \mathcal{O}_{n}$ 
for all $z\in {\Bbb T}$.
\end{enumerate}
\end{prop}
\begin{proof}
Clearly, $\gamma_{z}\lambda_{u}\gamma_{z}^{-1}=\lambda_{\gamma_{z}(u)}$ for all $z\in\T$. 
Thus condition (2) above is equivalent to $\gamma_{z}\lambda_{u}|_{\mathcal{F}_{n}}
=\lambda_{u}|_{\mathcal{F}_{n}}$ for all $z\in\T$.  Obviously, this holds if and only if 
$\lambda_{u}(\mathcal{F}_{n})\subseteq \mathcal{F}_{n}$. That is, $(1)$ is equivalent to $(2)$. 

It is an immediate consequence of Proposition 2.1 and Proposition 4.7 from \cite{CRS} that 
$\lambda_{u}|_{\mathcal{F}_{n}} = \lambda_{v}|_{\mathcal{F}_{n}}$ if and only if 
$vu^{*}\in \lambda_{u}(\mathcal{F}_{n})'\cap\mathcal{O}_{n}$. This gives $(2)$ is equivalent to $(3)$. 
\end{proof}

\begin{prop}\label{trivialrelativecommutant}
If $\lambda_{u}(\mathcal{F}_{n})\subseteq \mathcal{F}_{n}$ and 
$\lambda_{u}(\mathcal{F}_{n})'\cap\mathcal{F}_{n}={\Bbb C}1$, then $u\in \mathcal{F}_{n}$. 
\end{prop}
\begin{proof}
If $\lambda_{u}(\mathcal{F}_{n})\subseteq \mathcal{F}_{n}$ and 
$\lambda_{u}(\mathcal{F}_{n})'\cap\mathcal{F}_{n}={\Bbb C}1$, then 
$\lambda_{u}(\mathcal{F}_{n})'\cap \mathcal{O}_{n}={\Bbb C}1$ as well,  \cite[Theorem 1.1]{H}.  
As shown in \cite{CRS}, this implies that $u\in\F_n$.  
\end{proof}

\begin{prop}\label{oldlemma06}
Let $u$ be a unitary in $\OO_n$. Then $u=wv$ for some $w\in \lambda_{u}(\mathcal{F}_{n})'
\cap \mathcal{O}_{n}$ and a unitary $v\in \mathcal{F}_{n}$ if and only if there exists a unitary 
$y\in\lambda_{u}(\mathcal{F}_{n})'\cap \mathcal{O}_{n}$ such that $u\gamma_{z}(u^{*}) =  
y\gamma_{z}(y^{*})$ for all $z\in\T$. 
\end{prop}
\begin{proof}
If $u=wv$ for some $w\in \lambda_{u}(\mathcal{F}_{n})'\cap \mathcal{O}_{n}$ and $v\in\U( \mathcal{F}_{n})$, then $u\gamma_{z}(u^{*})=w\gamma_{z}(w^{*})$, and it suffices to put $y=w$. 

Conversely, if there exists a unitary $y\in\lambda_{u}(\mathcal{F}_{n})'\cap \mathcal{O}_{n}$ such 
that $u\gamma_{z}(u^{*}) =  y\gamma_{z}(y^{*})$ for all $z\in\T$ then $y^*u$ is fixed by all $\gamma_z$. 
Thus $y^*u\in\F_n$ and it suffices to put $w=y$ and $v=y^*u$. 
\end{proof}

From now on, we make  {\bf  a standing assumption} that $u\in\U(\OO_n)$ is such that 
\begin{equation}\label{assumption}
\lambda_{u}(\mathcal{F}_{n})\subseteq \mathcal{F}_{n} \;\; \text{and} \;\;
    {\rm dim}\lambda_{u}(\mathcal{F}_{n})'\cap\mathcal{F}_{n}<\infty. 
\end{equation}

As shown in \cite{H}, assumption (\ref{assumption}) above entails a number of important consequences, 
which we summarize as follows. 
\begin{itemize}
\item We also have ${\rm dim}\lambda_{u}(\mathcal{F}_{n})'\cap\mathcal{O}_{n}<\infty$. 
\item There exists a unitary group $\{u_{z}\}_{z\in{\Bbb T}}$ in the center of $\lambda_{u}(\mathcal{F}_{n})'
\cap \mathcal{F}_{n}$ such that ${\rm Ad}u_{z}(x)=\gamma_{z}(x)$ for all $x\in\lambda_{u}(\mathcal{F}_{n})'
\cap \mathcal{O}_{n}$. 
\item Minimal projections in 
$\lambda_{u}(\mathcal{F}_{n})'\cap \mathcal{F}_{n}$ are minimal in $\lambda_{u}(\mathcal{F}_{n})'
\cap \mathcal{O}_{n}$ as well. Thus $\lambda_{u}(\mathcal{F}_{n})'\cap \mathcal{O}_{n}$ 
contains a MASA consisting of projections in $\lambda_{u}(\mathcal{F}_{n})'\cap \mathcal{F}_{n}$. 
\end{itemize}

The proof of the following theorem is modelled after that of \cite[Lemma 1.11]{H}.

\begin{thm}\label{main1}
Let $u\in\U(\OO_n)$ be such that $\lambda_{u}(\mathcal{F}_{n})\subseteq \mathcal{F}_{n}$ and 
${\rm dim}\lambda_{u}(\mathcal{F}_{n})'\cap \mathcal{F}_{n}<\infty$. Then 
there exist unitaries  $w\in \lambda_{u}(\mathcal{F}_{n})'\cap \mathcal{O}_{n}$ and $v\in \OO_n$,  
and a unitary group $\{v_{z}\}_{z\in{\Bbb T}}\subseteq \lambda_{u}(\mathcal{F}_{n})'
\cap \mathcal{F}_{n}$ satisfying $u=wv$ and $\gamma_z(v)=v_z v$ for all $z\in{\Bbb T}$.  
\end{thm}
\begin{proof}
At first we note that $u\gamma_{z}(u^{*})u_{z}$ is a unitary group in $\lambda_{u}(\mathcal{F}_{n})'
\cap \mathcal{O}_{n}$. Indeed, 
$$ (u\gamma_{z_1}(u^*)u_{z_1})(u\gamma_{z_2}(u^*)u_{z_2}) = u\gamma_{z_1}(u^*)(\Ad u_{z_1})(u)
u_{z_1}\gamma_{z_2}(u^*)u_{z_2} $$
$$ = u\gamma_{z_1}(u^*)\gamma_{z_1}(u)\gamma_{z_1}(\gamma_{z_2}(u^*))u_{z_1}u_{z_2} = 
u\gamma_{z_1z_2}(u^*)u_{z_1z_2}. $$ 
Since ${\rm dim}\lambda_{u}(\mathcal{F}_{n})'\cap \mathcal{O}_{n}<\infty$, this unitary group may be 
diagonalized. On the other hand, $\lambda_{u}(\mathcal{F}_{n})'\cap \mathcal{O}_{n}$ contains a MASA 
composed of projections in $\lambda_{u}(\mathcal{F}_{n})'\cap \mathcal{F}_{n}$. Thus, 
there exists a unitary $w\in \lambda_{u}(\mathcal{F}_{n})'\cap \mathcal{O}_{n}$ such that  
$y_z:=w^*(u\gamma_z(u^*)u_z)w$ is a unitary group in $\lambda_{u}(\mathcal{F}_{n})'\cap \mathcal{F}_{n}$.
Since each $u_z$ is in the center of $\lambda_{u}(\mathcal{F}_{n})'\cap \mathcal{F}_{n}$, 
the unitary groups $\{y_{z}\}_{z\in\T}$ and $\{u_{z}\}_{z\in\T}$ commute. 

Set $v_z:=u_zy_z^*$, $z\in{\Bbb T}$, and $v:=w^*u$. Then $v_z$ is a unitary group in 
$\lambda_{u}(\mathcal{F}_{n})'\cap \mathcal{F}_{n}$ and 
$$ \gamma_z(v) = \gamma_z(w^*u) = u_z(u_z^*\gamma_z(w^*u)u^*w)w^*u = 
u_zy_z^*w^*u = v_zw^*u = v_zv. $$ 
for all $z\in{\Bbb T}$. This completes the proof. 
\end{proof}

We keep the notation from Theorem \ref{main1}, assuming that unitaries 
$w$, $v$ and $v_z$ have the properties described therein. Thus, in particular, 
$\lambda_{u}|_{\mathcal{F}_{n}}=\lambda_{v}|_{\mathcal{F}_{n}}$ by \cite[Proposition 2.1]{CRS}.
Consequently, ${\rm Ad}v\circ\varphi$ is an automorphism of $\lambda_{u}(\mathcal{F}_{n})'
\cap \mathcal{O}_{n}$, by \cite[Proposition 2.3 and Lemma 2.4]{CRS}. 

\begin{lem}\label{Xgroup}
With unitaries $u$, $v$, $v_z$ and $u_z$ as above, put 
$$
X_{z}:=({\rm Ad}v\circ\varphi)(u_{z})u_{z}^{*}v_{z}.
$$ 
Then $\{X_{z}\}_{z\in\T}$ is a unitary group in the center of $\lambda_{u}(\mathcal{F}_{n})'
\cap \mathcal{O}_{n}$, and we have 
$$
\gamma_{z}(v)=X_{z}u_{z}({\rm Ad}v\circ\varphi)(u_{z}^{*})v.
$$
\end{lem}
\begin{proof}
For each $x\in \lambda_{u}(\mathcal{F}_{n})'\cap \mathcal{O}_{n}$, we see that 
$$
u_{z}({\rm Ad}v\circ\varphi)(x)u_{z}^{*} = \gamma_z(v\varphi(x)v^*) = \gamma_z(v)\gamma_z(\varphi(x))
\gamma_z(v^*) = \gamma_{z}(v)\varphi(\gamma_{z}(x))\gamma_{z}(v)^{*}
$$
$$
= v_z v\varphi(u_z x u_z^*)v^* v_z^* = v_z v\varphi(u_z)v^*v\varphi(x)v^*v\varphi(u_z^*)v^*v_z^*  
$$
$$
= v_z\Ad(v\varphi(u_z)v^*)((\Ad v\circ\varphi)(x))v_z^*.  
$$
Hence, we have 
$$
\Ad(v_z^*u_z)((\Ad v\circ\varphi(x)) = \Ad((\Ad v\circ\varphi)(u_z))((\Ad v\circ\varphi)(x)). 
$$
Since ${\rm Ad}v\circ\varphi$ is an automorphism of $\lambda_{u}(\mathcal{F}_{n})'\cap \mathcal{O}_{n}$, 
this shows that 
\begin{equation}\label{oldlemma010}
\Ad(v_z^*u_z) = \Ad((\Ad v\circ\varphi)(u_z))
\;\; \text{on} \;\; \lambda_{u}(\mathcal{F}_{n})'\cap \mathcal{O}_{n}. 
\end{equation}
Consequently, $X_{z}$ belongs to the center of $\lambda_{u}(\mathcal{F}_{n})'\cap \mathcal{O}_{n}$. 

Now, $\{u_z\}_{z\in\T}$ and $\{v_z\}_{z\in\T}$ are commuting unitary groups, and both commute with 
$X_z$, by the above argument.  Therefore the unitary group $(\Ad v\circ\varphi)(u_z)=X_zu_zv_z^*$ 
commutes with both of them.  Consequently, $X_z$ being a product of three mutually commuting unitary groups 
itself is a unitary group. 

The final claim of the lemma now follows from the fact that $\gamma_z(v)=v_zv$. 
\end{proof}

Before proceeding further, we introduce the following notation. 
For $x\in \lambda_{u}(\mathcal{F}_{n})'\cap \mathcal{O}_{n}$ and $k\in\N$, we set 
\begin{equation}\label{xk}
x^{(k)}:=x({\rm Ad}v\circ\varphi)(x)({\rm Ad}v\circ\varphi)^{2}(x)\cdots({\rm Ad}v\circ\varphi)^{k-1}(x).
\end{equation}

\begin{lem}\label{vzk}
With unitaries $u$, $v$, $v_z$ and $u_z$ as above, and $x\in \lambda_{u}(\mathcal{F}_{n})'\cap \mathcal{O}_{n}$, for all $g\in\U(\OO_n)$, $z\in\T$ and $k\in\N$ we have the following identities. 
\begin{enumerate}
\item[(i)] $\gamma_{z}(v_{k})=v_{z}^{(k)}v_{k}$, 
\item[(ii)] $({\rm Ad}g\circ\varphi)^{k}(x)=g_{k}\varphi^{k}(x)g_{k}^{*}$, 
\item[(iii)] $v_{z}^{(k)}=X_{z}^{(k)}u_{z}({\rm Ad}v\circ\varphi)^{k}(u_{z}^{*})$, 
\item[(iv)] $(gv)_k = g^{(k)}v_k$.    
\end{enumerate}
\end{lem}
\begin{proof}
In all three cases, we proceed by induction on $k$. 

\smallskip\noindent
Ad (i). Case $k=1$ is the identity $\gamma_z(v_1)=\gamma_z(v)=
v_zv=v_z^{(1)}v_1$ from Theorem \ref{main1}. For the inductive step, we calculate
$$
\gamma_z(v_{k+1}) = \gamma_z(v_k\varphi^k(v)) = \gamma_z(v_k)\varphi^k(\gamma_z(v)) = 
v_z^{(k)}v_k\varphi^k(v_z)v_k^*v_k\varphi^k(v) = v_z^{(k+1)}v_{k+1}. 
$$
In this calculation we used identity (ii) of the present lemma, whose proof does not depend on (i). 

\noindent
Ad (ii). Case $k=1$ is clear. For the inductive step, we have
$$
(\Ad g\circ\varphi)^{k+1} = (\Ad g\circ\varphi)(g_k\varphi^k(x)g_k^*) = g\varphi(g_k)\varphi^{k+1}(x)
\varphi(g_k^*)g^* = g_{k+1}\varphi^{k+1}(x)g_{k+1}^*. 
$$

\noindent
Ad (iii). Case $k=1$ is clear. For the inductive step, we see that
$$
v_z^{(k+1)} = v_z^{(k)}(\Ad v\circ\varphi)^k(v_z) = X_{z}^{(k)}u_{z}({\rm Ad}v\circ\varphi)^{k}(u_{z}^{*}) 
(\Ad v\circ\varphi)^k(v_z) 
$$
$$
=   X_{z}^{(k)}u_{z}({\rm Ad}v\circ\varphi)^{k}(u_{z}^{*}v_z) = 
X_{z}^{(k)}u_{z}({\rm Ad}v\circ\varphi)^{k}(X_z(\Ad v\circ\varphi)(u_z^*)) 
$$
$$
= X_{z}^{(k)}u_{z}({\rm Ad}v\circ\varphi)^{k}(X_z)(\Ad v\circ\varphi)^{k+1}(u_z^*) = 
X_{z}^{(k+1)}u_{z}({\rm Ad}v\circ\varphi)^{k+1}(u_{z}^{*}).  
$$

\noindent
Ad (iv). Case $k=1$ is clear. For the inductive step, we calculate using part (ii) above, 
$$
(gv)_{k+1} = (gv)_k\varphi^k(gv) = g^{(k)}v_k\varphi^k(gv) = g^{(k)}(v_k\varphi^k(g)v_k^*)v_k
\varphi^k(v) = g^{(k+1)}v_{k+1},   
$$
and this completes the proof. 
\end{proof}

The following lemma provides a key step in the proof of our second main result, Theorem \ref{main2}, below. 
We continue keeping the notation of Theorem \ref{main1}. Here we remark that since $v=w^*u$, 
$w\in\lambda_u(\F_n)'\cap\OO_n$, and $\Ad u\circ\varphi$ is an automorphism of $\lambda_u(\F_n)'\cap
\OO_n$, we see that $\Ad v\circ\varphi = \Ad w^* \circ(\Ad u\circ\varphi)$ is an automorphism of $\lambda_u(\F_n)'\cap\OO_n$ as well. We also note that for each positive integer $k$, 
$\{X_{z}^{(k)}\}_{z\in{\Bbb T}}$ is a unitary group in the center of $\lambda_u(\F_n)'\cap\OO_n$. 

\begin{lem}\label{Uk}
With unitaries $u$, $v$, $v_z$ and $u_z$ as above, 
there exist a positive integer $k$ and a unitary $U\in \lambda_{u}(\mathcal{F}_{n})'\cap \mathcal{O}_{n}$ 
such that  
$$
(\Ad v\circ\varphi)^k(x) = \Ad U(x) \;\;\; \text{for all} \; x\in\lambda_u(\F_n)'\cap\OO_n. 
$$
Then $X_z^{(k)}=1$. Furthermore, for such $U$ and $k$, we have $U^{*}v_{k}\in \mathcal{F}_{n}$.
\end{lem}
\begin{proof}
Since $\Ad v\circ\varphi$ is an automorphism of a finite dimensional $C^*$-algebra $\lambda_u(\F_n)'\cap\OO_n$, 
its restricts to the center has finite order. Thus there exists a positive integer $k$ and a unitary 
$U\in \lambda_{u}(\mathcal{F}_{n})'\cap \mathcal{O}_{n}$ such that $(\Ad v\circ\varphi)^k=\Ad U$ on 
$\lambda_u(\F_n)'\cap\OO_n$. We claim that $U^{*}v_{k}\in \mathcal{F}_{n}$. 

Indeed, by Lemma \ref{vzk}, for all $z\in\T$ we have 
$$
\gamma_{z}(v_{k})=v_{z}^{(k)}v_{k}=X_{z}^{(k)}u_{z}({\rm Ad}v\circ\varphi)^{k}(u_{z}^{*})
v_{k}=X_{z}^{(k)}u_{z}Uu_{z}^{*}U^{*}v_{k}=X_{z}^{(k)}\gamma_{z}(U)U^{*}v_{k}, 
$$
and this yields 
\begin{equation}\label{gU}
\gamma_{z}(U^{*}v_{k})=X_{z}^{(k)}U^{*}v_{k}.  
\end{equation}
Since $\{X_{z}^{(k)}\}_{z\in{\Bbb T}}$ is a unitary group in the center of $\lambda_u(\F_n)'\cap\OO_n$, 
there exists a partition of unity $1=\sum_i p_i$ in ${\mathcal Z}(\lambda_u(\F_n)'\cap\OO_n)$ and integers 
$k_i$ such that 
$$
X_{z}^{(k)}=\sum_{i}z^{k_{i}}p_{i}.
$$ 
We have $({\rm Ad}v\circ\varphi)^{k}(p_i)=Up_iU^*=p_i$ for all $i$. Combining this with part (ii) of Lemma \ref{vzk}, we get 
\begin{equation}\label{vp}
v_{k}^{*}p_{i}v_{k}=\varphi^{k}(p_{i}). 
\end{equation}
We want to show that $k_i=0$ for all $i$. Suppose for a moment this is not the case and let $k_i>0$ for 
some $i$. We set 
$K:=p_{i}U^{*}v_{k}(S_{1}^{*})^{k_{i}}$. Since $p_i$ being in $\Z(\rc)$ belongs to $\F_n$ as well, 
it follows from identity (\ref{gU}) above that  
$$
\gamma_z(K)=\gamma_z(p_iU^*v_k(S_1^*)^{k_i}) = p_iX_z^{(k)}U^*v_k\gamma_z((S_1^*)^{k_i}) = 
z^{k_i}p_iU^*v_k(z^{-k_i}(S_1^*)^{k_i}) = K. 
$$ 
Hence $K$ belongs to $\mathcal{F}_{n}$. We have $KK^*=p_i$. On the other hand, using identity 
(\ref{vp}) we get 
$$
K^{*}K=S_{1}^{k_{i}}v_{k}^{*}p_{i}v_{k}(S_{1}^{*})^{k_{i}}
=S_{1}^{k_{i}}\varphi^{k}(p_{i})(S_{1}^{*})^{k_{i}}
=\varphi^{k+k_{i}}(p_{i})S_{1}^{k_{i}}(S_{1}^{*})^{k_{i}}.
$$ 
It easily follows that $\tau(KK^{*})>\tau(K^{*}K)$, which is a contradiction. A similar argument applies 
in the case $k_i<0$. Hence $k_i=0$ for all $i$ and thus $X_z^{(k)}=1$. Now, identity (\ref{gU}) implies that 
$U^*v_k$ is fixed by the gauge action and hence belongs to $\F_n$. 
\end{proof}

Now, we are ready to prove the second main result of this paper. 

\begin{thm}\label{main2}
Let $u\in\U(\OO_n)$ be such that $\lambda_{u}(\mathcal{F}_{n})\subseteq \mathcal{F}_{n}$ and 
${\rm dim}\lambda_{u}(\mathcal{F}_{n})'\cap \mathcal{F}_{n}<\infty$. Then 
there exist a positive integer $k$ and unitaries $W\in \lambda_{u}(\mathcal{F}_{n})'\cap 
\mathcal{O}_{n}$ and $V\in \mathcal{F}_{n}$ such that $u_{k}=WV$. 
\end{thm}
\begin{proof}
By Theorem \ref{main1} and Lemma \ref{Uk}, there exist unitaries $w,U\in 
\lambda_{u}(\mathcal{F}_{n})'\cap \mathcal{O}_{n}$, a unitary group $\{v_{z}\}_{z\in\T}$ in 
$\lambda_{u}(\mathcal{F}_{n})'\cap \mathcal{F}_{n}$ and a positive integer $k$ satisfying 
$u=wv$, $\gamma_{z}(v)=v_{z}v$, $U^{*}v_{k}\in \mathcal{F}_{n}$.
By part (iv) of Lemma \ref{vzk}, we have $w^{(k)}v_{k} = u_k$. Thus 
to complete the proof, it suffices to put $W:=w^{(k)}U$ and $V:=U^*v_k$. 
\end{proof}

It was observed in \cite{CRS} (just above Remark 4.4) that if $\lambda_u(\F_n)\subseteq\F_n$ and 
$\rc=\C 1$ then $u\in\F_n$. The following corollary gives a sharp strengthening of that result.

\begin{cor}\label{inner}
Let $u$ be a unitary in $\OO_n$. 
If $\lambda_{u}(\mathcal{F}_{n})\subseteq \mathcal{F}_{n}$,  
${\rm dim}\lambda_{u}(\mathcal{F}_{n})'\cap \mathcal{F}_{n}<\infty$  and 
the automorphism $\Ad u\circ\varphi$ of $\rc$ is inner, then 
there exist a unitary $w\in\rc$ and a unitary $v\in\F_n$ such that $u=wv$, and hence also  
$\lambda_u|_{\F_n} = \lambda_v|_{\F_n}$. 
In particular, this is the case whenever $\lambda_u(\F_n)'\cap\OO_n$ is a factor. 
\end{cor}

\begin{remark}\label{equivalentinner} \rm
The assumption in Corollary \ref{inner} above that the automorphism $\Ad u\circ\varphi$ of $\rc$ be inner, 
is equivalent to demanding existence of a unitary $g$ in the relative commutant $\rc$ such that 
$$ \lambda_{gu}(\F_n)'\cap\OO_n = \lambda_{gu}(\OO_n)'\cap\OO_n. $$
Indeed, if $\Ad u\circ\varphi$ is inner then $\Ad gu\circ\varphi =\id$ for a suitable unitary $g$ in $\rc$. Hence 
$\rc = \lambda_{gu}(\F_n)'\cap\OO_n = \lambda_{gu}(\OO_n)'\cap\OO_n$. Conversely, if 
$\lambda_{gu}(\F_n)'\cap\OO_n = \lambda_{gu}(\OO_n)'\cap\OO_n$ then $\Ad gu\circ\varphi=\id$ 
on $\rc = \lambda_{gu}(\F_n)'\cap\OO_n$, and hence $\Ad u\circ\varphi$ is inner. 
\hfill$\Box$
\end{remark}

\begin{remark}\label{nonouter} \rm
We remark that the implication in Corollary \ref{inner} above cannot be reversed. In fact, there exist 
unitaries $u\in\F_n$ such that $\rc$ is finite dimensional and the automorphism $\Ad u\circ\varphi$ 
is outer on $\rc$. For example, take 
$$ u = S_{22}S_{11}^* + S_{12}S_{22}^* + S_{11}S_{12}^* + P_{21}, $$ 
a permutative unitary in $\F_2$.  Then $\Ad u\circ\varphi$ is outer on $\lambda_u(\F_2)'\cap\OO_2$. 
For otherwise let $h$ be a unitary in $\lambda_u(\F_2)'\cap\OO_2$ such that $\Ad u\circ\varphi = \Ad h$ on $\lambda_u(\F_2)'\cap\OO_2$. Then $\Ad u\circ\varphi(h)=h$ and thus $h\in\lambda_u(\OO_2)'\cap\OO_2$. 
But it can be shown 
that $\lambda_u$ is irreducible on $\OO_2$ (e.g., see \cite{CSJMP}, where this endomorphism is 
denoted $\rho_{142}$), and hence $\lambda_u(\OO_2)'\cap\OO_2=\C 1$. Thus $h$ is a scalar and 
consequently $\Ad u\circ\varphi$ is identity on $\lambda_u(\F_2)'\cap\OO_2$. This however is not 
the case, since one can calculate directly that $\Ad u\circ\varphi$ permutes $P_1$ and $P_2$, 
and both these projections are in $\lambda_u(\F_2)'\cap\OO_2$. 
\hfill$\Box$
\end{remark}

We want to elaborate a little bit the statement of Theorem \ref{main2} above. 
We continue keeping our standing assumption (\ref{assumption}). 

\begin{lem}\label{masa}
Let $\alpha$ be an automorphism of $\lambda_u(\F_n)'\cap\OO_n$ and 
let $k\in N$ be such that $\alpha^k$ acts trivially on $\Z(\rc)$. 
Then there exists a MASA $D$ of $\lambda_u(\F_n)'\cap\F_n$ and a unitary $g$ in $\rc$ such that 
\begin{enumerate}
\item[(i)] $(\Ad g\circ\alpha)^k =\id$, and 
\item[(ii)] $(\Ad g\circ\alpha)(D)=D$. 
\end{enumerate}
\end{lem}
\begin{proof}
Automorphism $\alpha$ permutes the finitely many minimal central projections of $\rc$. Write this 
permutation as a product of disjoint cycles. Clearly, it suffices to prove the lemma for each cycle 
separately. Thus we may simply assume that $\alpha$ acts transitively on minimal projections $p_1,p_2,\ldots,p_l$ 
in $\Z(\rc)$, so that $\alpha(p_i)=p_{i+1}$, with $p_{l+1}=p_1$. Let $\{e_{r,s}^{(i)}\}$ be 
matrix units of the full matrix algebra $p_i(\rc)$, such that all $e_{r,r}^{(i)}$ are in $\lambda_u(\F_n)'\cap\F_n$. 
Then $D:=\span\{e_{r,r}^{(i)}\}$ is a MASA in $\lambda_u(\F_n)'\cap\F_n$. Since 
$p_i(\rc) \cong p_{i+1}(\rc)$, we can find a unitary $g_i\in p_{i+1}(\rc)$ such that $(\Ad g_i\circ\alpha)
(e_{r,s}^{(i)}) =  e_{r,s}^{(i+1)}$. Setting $g:=\sum_{i=1}^l g_i$ we obtain the desired result. 
\end{proof}

\begin{lem}\label{avk}
Let $u\in\U(\OO_n)$ be such that $\lambda_{u}(\mathcal{F}_{n})\subseteq \mathcal{F}_{n}$ and 
${\rm dim}\lambda_{u}(\mathcal{F}_{n})'\cap \mathcal{F}_{n}<\infty$. Then 
there exist a positive integer $k$, a unitary $g\in \lambda_{u}(\mathcal{F}_{n})'\cap \mathcal{O}_{n}$, 
and a unitary group $\{d_z\}_{z\in\T}\subseteq \lambda_{u}(\mathcal{F}_{n})'\cap \mathcal{F}_{n}$ 
such that $(gv)_{k}\in \mathcal{F}_{n}$ and  $\gamma_{z}(gv)=d_{z}gv$. 
\end{lem}
\begin{proof}
Put $\alpha:=\Ad v\circ\varphi$, and let $g$ and $k$ be as in Lemma \ref{masa}. Then we have 
$$ (\Ad gv \circ \varphi)^k = \id, $$ 
and thus 
$$ \Ad v_k \circ \varphi^k = (\Ad v \circ \varphi)^k = \Ad (g^{(k)})^* $$ 
by parts (ii) and  (iv) of Lemma \ref{vzk}. 
Then arguing as in the proof of Lemma \ref{Uk} (with ${g^{(k)}}^{*}$ playing the role of $U$), we get 
$$
(gv)_{k}=g^{(k)}v_{k}\in \mathcal{F}_{n}.
$$
Now, let $D$ be a MASA as in Lemma \ref{masa}. For all $x\in D$ and $z\in\T$, 
we see that 
$$
gv\varphi(x)v^{*}g^{*} = \gamma_{z}(gv\varphi(x)v^{*}g^{*}) = \gamma_z(g)v_zv\varphi(x)
v^*v_z^*\gamma_z(g^*)  
$$ 
$$
= (\gamma_z(g)v_zg^*)(gv\varphi(x)v^*g^*)(\gamma_z(g)v_zg)^*,  
$$
which implies that $\gamma_z(g)v_zg^*$ is in the commutant of MASA $D$, and hence in $D$ itself. 
Set $d_{z}=\gamma_{z}(g)v_{z}g^{*}$, a unitary in $D$. Now, $d_z=u_zgu_z^*v_zg^*$ implies 
$u_z^*d_z = g(u_z^*v_z)g^*$. Since $\{u_z\}_{z\in\T}$ and $\{v_z\}_{z\in\T}$ are commuting 
unitary groups, so is $\{u_z^*d_z\}_{z\in\T}$, and consequently also is $\{d_z\}_{z\in\T}$. 
Finally, we see that $\gamma_z(gv)=\gamma_z(g)v_zv=d_zgv$. 
\end{proof}

Now, we are ready to prove the following result. 

\begin{thm}\label{main3}
Let $u\in\U(\OO_n)$. 
If $\lambda_{u}(\mathcal{F}_{n})\subseteq \mathcal{F}_{n}$ and 
${\rm dim}\lambda_{u}(\mathcal{F}_{n})'\cap \mathcal{F}_{n}<\infty$, then there exists a unitary 
$W\in \lambda_{u}(\mathcal{F}_{n})'\cap \mathcal{O}_{n}$ satisfying the following. 
\begin{enumerate}
\item 
There exists a unitary group $\{d_{z}\}_{z\in\T}\subseteq \lambda_{u}(\mathcal{F}_{n})'\cap 
\mathcal{F}_{n}$ such that $\gamma_{z}(Wu)=d_{z}Wu$ for all $z\in {\Bbb T}$.
\item There exists a positive integer $k$ such that $(Wu)_{k}\in \mathcal{F}_{n}$. 
\end{enumerate}
\end{thm}
\begin{proof}
Let $u=wv$ be a factorization as in Theorem \ref{main1}, and let $k\in\N$ and 
$g\in\lambda_u(\F_n)'\cap\OO_n$ be as in Lemma \ref{avk} above. Then setting $W:=gw^*$ gives the claim. 
\end{proof}


\section{The criterion and examples}

In  this section, we give a dynamic characterization of those unitaries $u\in\OO_n$ satisfying our standing 
assumptions which either belong to $\F_n$ (Theorem \ref{main4}) or admit a unitary $v\in\F_n$ such that 
$\lambda_u|_{\F_n} = \lambda_v|_{\F_n}$ (Corollary \ref{maincriterion}). Before proving these results, 
we still need one technical lemma about the structure of the relative commutants. We keep our standing 
assumptions (\ref{assumption}). 

\begin{lem}\label{xztechnical}
There exist a unitary group $\{q_z\}_{z\in\T}$ in $\Z(\rc)$ such that 
$$ X_z = q_z(\Ad v\circ\varphi)(q_z^*). $$
\end{lem}
\begin{proof}
Since $\Ad v\circ\varphi$ restricts to an automorphism of $\Z(\rc)$, there exist minimal projections 
$p_i^{(j)}$, $j=1,\ldots,N$, $i=1,\ldots,n_j$, in $\Z(\rc)$ such that 
$$ \Z(\rc) =\bigoplus_{j=1}^N\bigoplus_{i=1}^{n_j}p_i^{(j)} $$
and 
$$ (\Ad v\circ\varphi)(p_i^{(j)})=p_{i+1}^{(j)} \;\;\text{for } i<n_j, \;\;\text{and } 
(\Ad v\circ\varphi)(p_{n_i}^{(j)})=p_{1}^{(j)}. $$
Then $X_z$ from Lemma \ref{Xgroup} can be written as
$$  X_z = \sum_{j=1}^N\sum_{i=1}^{n_j}z^{m_i^{(j)}}p_i^{(j)}, $$
for some $m_i^{(j)}\in\N$. Now, let $k\in\N$ be such that $\Ad v\circ\varphi$ is an inner automorphism 
of $\rc$. Then  
$$ X_z^{(k)}= X_z(\Ad v\circ\varphi)(X_z)(\Ad v\circ\varphi)^2(X_z)\ldots(\Ad v\circ\varphi)^{k-1}(X_z) =1 $$
by Lemma \ref{Uk}. Since each $n_j$ divides $k$, this implies that 
$$ \sum_{i=1}^{n_j}m_i^{(j)} = 0 $$ 
for each $j=1,\ldots,N$. Now, we want to define $q_z$ as follows, 
$$ q_z = \sum_{j=1}^N\sum_{i=1}^{n_j}z^{r_i^{(j)}}p_i^{(j)}, $$
for suitable chosen integers $r_i^{(j)}$, so that $X_z = q_z(\Ad v\circ\varphi)(q_z^*)$. 
To this end, it suffices to put 
$$ \begin{aligned}
r_1^{(j)} & ;= 0, \;\; j=1,\ldots,N, \\
r_k^{(j)} & := \sum_{r=2}^k m_r^{(j)}, \;\; j=1,\ldots,N,\; k=2,\ldots,n_j. 
\end{aligned} $$
\end{proof}

\begin{thm}\label{main4}
Let $u\in\U(\OO_n)$ be such that $\lambda_{u}(\mathcal{F}_{n})\subseteq \mathcal{F}_{n}$ and 
${\rm dim}\lambda_{u}(\mathcal{F}_{n})'\cap \mathcal{F}_{n}<\infty$. Put $\alpha:=\Ad u\circ\varphi$. 
If $\alpha$ satisfies the following two conditions:
\begin{enumerate}
\item[(i)] $\alpha(\frc) = \frc$, and 
\item[(ii)] $\alpha|_{\frc}$ preserves the $\tau$-trace, 
\end{enumerate}
then $u\in\F_n$. 
\end{thm}
\begin{proof}
At first, we observe that there exists a unitary group $\{u'_z\}_{z\in\T}$ in $\Z(\frc)$ such that 
$\Ad u'_z(x)=\gamma_z(x)$ for all $x\in\rc$ and $\gamma_z(u)=u'_z\alpha({u'_z}^*)u$. Indeed, 
it suffices to put $u'_z:=q_zu_z$, with $q_z$ as in Lemma \ref{xztechnical} above. Then 
$\alpha(u'_z)\in\Z(\frc)$ by condition (i) of the theorem, and hence $\{u'_z\alpha({u'_z}^*)\}_{z\in\T}$ 
is a unitary group. Thus, $u'_z\alpha({u'_z}^*) = \sum z^{k_j}p_j$ for some integers $k_j$ and a partition 
of unity by projections $p_j$ from  $\Z(\frc)$. 

Now, we claim that $p_j=0$ whenever $k_j\neq 0$. To this end, suppose first that $k_j>0$ for some index $j$, 
and put $R:=p_{k_j}u(S_1^*)^{k_j}$. We have $\gamma_z(R)=R$ for all $z\in\T$, and thus $R\in\F_n$. 
However, an easy calculation shows that $RR^*=p_{k_j}$ and $R^*R=\varphi^{k+1}(\alpha^{-1}(p_{k_j}))
S_1^k(S_1^*)^k$. In view of condition (ii) of the theorem, this would imply $\tau(RR^*)\neq\tau(R^*R)$ if 
$p_j\neq 0$, a contradiction. Therefore $p_j=0$ for all $k_j>0$. A similar argument shows that 
$p_j=0$ if $k_j<0$. 

Consequently, $u'_z\alpha({u'_z}^*) = 1$. But this gives $\gamma_z(u)=u$ for all $z\in\T$. Hence $u\in\F_n$ 
and the theorem is proved. 
\end{proof}
We note that Theorem \ref{main4} gives a necessary and sufficient condition for $u\in\F_n$, since the 
reverse implication is trivial. Likewise, Corollary \ref{kmain4} below,  gives a necessary and sufficient 
condition for $u_k\in\F_n$. 

\begin{cor}\label{kmain4}
Let $u\in\U(\OO_n)$ be such that $\lambda_{u}(\mathcal{F}_{n})\subseteq \mathcal{F}_{n}$ and 
${\rm dim}\lambda_{u}(\mathcal{F}_{n})'\cap \mathcal{F}_{n}<\infty$. Put $\alpha:=(\Ad u\circ\varphi)^k$, 
for some positive integer $k$. If $\alpha$ satisfies the following two conditions:
\begin{enumerate}
\item[(i)] $\alpha(\frc) = \frc$, and 
\item[(ii)] $\alpha|_{\frc}$ preserves the $\tau$-trace, 
\end{enumerate}
then $u_k\in\F_n$. 
\end{cor}

Now, we are ready to give the following decomposability criterion. 

\begin{cor}\label{maincriterion}
Let $u\in\U(\OO_n)$ be such that $\lambda_{u}(\mathcal{F}_{n})\subseteq \mathcal{F}_{n}$ and 
${\rm dim}\lambda_{u}(\mathcal{F}_{n})'\cap \mathcal{F}_{n}<\infty$. Put $\alpha:=\Ad u\circ\varphi$. 
Then the following two conditions are equivalent:
\begin{enumerate}
\item[(1)]  There exist unitaries $w\in\lambda_{u}(\mathcal{F}_{n})'\cap\mathcal{O}_{n}$ and 
$v\in \mathcal{F}_{n}$ such that $u=wv$.  
\item[(2)] For each minimal projection $p\in\Z(\rc)$ there exists a $\tau$-preserving isomorphism 
$$ p(\frc) \cong \alpha(p)(\frc). $$
\end{enumerate}
\end{cor}

Now, we show how to construct examples of endomorphisms $\lambda_u$ of $\OO_n$ globally 
preserving the core UHF-subalgebra $\F_n$ but such that no unitary $v\in\F_n$ exists for which 
$\lambda_u|_{\F_n}=\lambda_v|_{\F_n}$. 

To begin with, take two non-zero, orthogonal projections $q_1, q_2$ in $\F_n$ such that 
$\tau(q_2)/\tau(q_1)=n^r$ for some non-zero integer $r$. Let $A_1$ be a partial isometry in $\OO_n^{(-r)}$ 
with domain projection $\varphi(q_1)$ and range projection $q_2$. Likewise, let $A_2$ be a partial isometry 
in $\OO_n^{(r)}$ with domain projection $\varphi(q_2)$ and range projection $q_1$. Finally, let $A_3$ be a 
partial isometry in $\F_n$ with domain projection $1-\varphi(q_1)-\varphi(q_2)$ and range projection $1-q_1-q_2$. 
Put $u:=A_1+A_2+A_3$. Then $u$ is a unitary in $\OO_n$ such that 
\begin{equation}\label{aduq}
\Ad u\circ\varphi(q_1)=q_2 \;\;\; \text{and} \;\;\; \Ad u\circ\varphi(q_2)=q_1. 
\end{equation}
Now, $u\gamma_z(u^*)=z^rq_1 + z^{-r}q_2 +1-q_1-q_2$ belongs to $\span\{1,q_1,q_2\}$, 
and $\span\{1,q_1,q_2\}\subseteq\rc$ by \cite[Proposition 2.3]{CRS} and (\ref{aduq}) above.  
Thus $\lambda_u(\F_n)\subseteq\F_n$ by Proposition \ref{3conditions} above. 

More generally, Let $1=\sum q_j$ be a partition of unity by projections in $\OO_n$. Let $u$ be any unitary in 
$\OO_n$ such that $\Ad u\circ\varphi$ permutes projections $\{q_j\}$ and for each $j$ there is a $k_j\in\Zb$ 
such that $q_j u\in\OO_n^{(k_j)}$. Then $u\gamma_z(u^*)\in\span\{q_j\}\subseteq\rc$ for 
all $z\in\T$. This simple construction gives a large class of examples of unitaries $u\in\OO_n\setminus\F_n$ 
such that $\lambda_u(\F_n)\subseteq\F_n$. However, to verify the conditions of 
Corollary \ref{maincriterion} one needs more detailed information on the relative 
commutants $\lambda_u(\F_n)'\cap\F_n \subseteq \lambda_u(\F_n)'\cap\OO_n$. Exact determination 
of these relative commutants is rather difficult and does not seem possible in general, despite the 
identity from \cite[Proposition 2.3]{CRS}. However, it is quite doable in concrete cases. 

Now, we illustrate the above discussion with two concrete examples in $\OO_2$. In these examples, along with 
the main algebra $C^*(S_1,S_2)\cong\OO_2$, we consider its other subalgebras, also isomorphic to $\OO_2$. 
For example, if $T_1,T_2$ are isometries in $C^*(S_1,S_2)$ generating a copy of $\OO_2$, then we use 
subscript $T$ along with the standard notation to indicate that the object comes from $C^*(T_1,T_2)$ 
and its generators. Thus $\varphi_T$ denotes the usual shift on $C^(T_1,T_2)$, that is a map 
$\varphi:C^*(T_1,T_2) \to C^*(T_1,T_2)$ such that $\varphi(x)=T_1 xT_1^* + T_2 xT_2^*$. 
Similarly, $\D_T$ denotes the diagonal MASA of $C^*(T_1,T_2)$, and so on. The proof of one 
technical lemma needed in Example \ref{example1} is given afterwards. 


\begin{ex}\label{example2}
\rm 
Take $q_1=P_{11}$, $q_2=P_{222}$, and set 
$$ \begin{aligned}
A_1 & =  S_{2221}S_{111}^* + S_{2222}S_{211}^* \\
A_2 & = S_{111}S_{1222}^* + S_{112}S_{2222}^* \\
A_3 & = S_{1222}S_{2221}^* + S_{211}S_{112}^* + P_{121} + P_{1221} + P_{212} + P_{221}
\end{aligned} $$
We note that unitary $u:=A_1+A_2+A_3$ falls within the class of polynomial unitaries considered in 
\cite[Section 5]{CRS}. In particular, its graph $E_u$, as defined therein, admits the $\{-1,0,+1\}$ labelling:

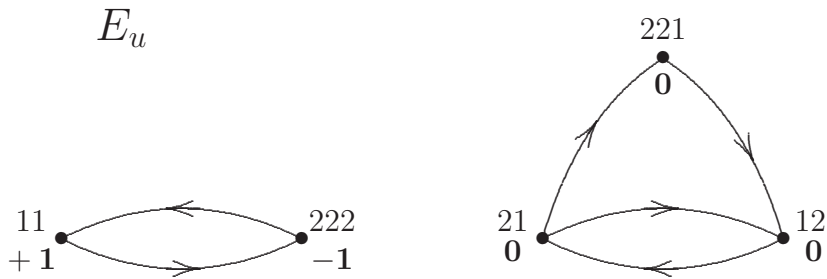
\[ \beginpicture
\setcoordinatesystem units <0.8cm,0.8cm>
\setplotarea x from -5 to 1, y from -2 to 2
\put {$\bullet$} at -4 0
\put {$\bullet$} at -8 0

\put {$\bullet$} at 0 0
\put {$\bullet$} at 4 0
\put {$\bullet$} at 2 3

\setquadratic
\plot -8 0  -6 0.5  -4 0 /
\plot -8 0  -6 -0.5  -4 0 /

\plot 0 0  2 0.5  4 0 /
\plot 0 0  2 -0.5  4 0 /

\plot 0 0  0.8 1.8  2 3 /
\plot 4 0 3.2 1.8  2 3 /

\arrow <0.35cm> [0.2,0.6] from -5.7 0.5 to -6.25 0.5
\arrow <0.35cm> [0.2,0.6] from -6.25 -0.5 to -5.7 -0.5

\arrow <0.35cm> [0.2,0.6] from 1.7 0.5 to 2.25 0.5
\arrow <0.35cm> [0.2,0.6] from 2.25 -0.5 to 1.7 -0.5

\arrow <0.35cm> [0.2,0.6] from 0.765 1.75 to 0.9 1.95
\arrow <0.35cm> [0.2,0.6] from 3.45 1.4 to 3.5 1.3

\put{$11$} at -8.5 0.3
\put{${\bf +1}$} at -8.55 -0.35
\put{$222$} at -3.5 0.3
\put{${\bf -1}$} at -3.5 -0.35

\put{$21$} at -0.5 0.3
\put{${\bf 0}$} at -0.5 -0.2
\put{$12$} at 4.45 0.3
\put{${\bf 0}$} at 4.5 -0.2
\put{$221$} at 2 3.5
\put{${\bf 0}$} at 2 2.58

\put{\Large $E_u$} at -7 3.5
\endpicture \] 
This labelled graph satisfies the path condition defined in \cite[p.~616]{CRS}, 
and this is an alternative way of showing that $\lambda_u(\F_2)\subseteq\F_2$. 

Now, we have $P_{11}\OO_2 P_{11} \cong \OO_2=C^*(T_1,T_2)$, under the isomorphism 
sending $T_1$ to $S_{111}S_{11}^*$ and $T_2$ to $S_{112}S_{11}^*$. Similarly, 
$P_{222}\OO_2 P_{222} \cong \OO_2=C^*(R_1,R_2)$, under the isomorphism sending 
$R_1$ to $S_{2221}S_{222}^*$ and $R_2$ to $S_{2222}S_{222}^*$. Then an easy calculation shows that 
$$ \begin{aligned}
\Ad u\circ\varphi(T_j) & = \varphi_R(R_j), \\ 
\Ad u\circ\varphi(R_j) & = \varphi_T(T_j), 
\end{aligned} $$
for $j=1,2$. Consequently, the restriction of $(\Ad u\circ\varphi)^2$ 
to $P_{11}\OO_2 P_{11}$ is conjugate to $\varphi_R^2$. Likewise, the restriction of $(\Ad u\circ\varphi)^2$ 
to $P_{222}\OO_2 P_{222}$ is conjugate to $\varphi_T^2$. This immediately implies 
$$ \begin{aligned}
\lambda_u(\F_2)'\cap P_{11}\OO_2 P_{11} & \subseteq \bigcap_{k=1}^\infty (\Ad u\circ\varphi)^{2k} 
   (P_{11}\OO_2 P_{11}) = \C P_{11}, \\ 
\lambda_u(\F_2)'\cap P_{222}\OO_2 P_{222} & \subseteq \bigcap_{k=1}^\infty (\Ad u\circ\varphi)^{2k} 
   (P_{222}\OO_2 P_{222}) = \C P_{222}. 
\end{aligned} $$
That is, both $P_{11}$ and $P_{222}$ are minimal projections in $\rcc$. 
One easily checks that $\Ad u\circ\varphi(S_{11}S_{222}^*)=S_{222}S_{11}^*$. Thus  $S_{11}S_{222}^*$ 
is in $\rcc$, and we see that $(P_{11}+P_{222})\rcc(P_{11}+P_{222})\cong\Mb_2(\C)$. 
We remark that the restriction 
of $\Ad u\circ\varphi$ to $(P_{11}+P_{222})\OO_2(P_{11}+P_{222})$ is conjugate to endomorphism 
$\rho_{1342}$ from \cite{CSJMP}. Let 
$$ w:=S_{11}S_{222}^* + S_{222}S_{11}^* + 1-P_{11}-P_{222}. $$
Then $w$ is a unitary in $\rcc$ such that $w^*u\in\F_2$. 
\hfill$\Box$
\end{ex}


\begin{ex}\label{example1}{\rm

Take $q_1=P_{1}$, $q_2=P_{21}$, and set 
$$ \begin{aligned}
A_1 & = S_{211}S_{21}^* + S_{2121}S_{112}^* + S_{2122}S_{111}^*,  \\
A_2 & = S_{12}S_{121}^* + S_{11}S_{221}^*,  \\
A_3 & = S_{221}S_{122}^* + P_{222}. 
\end{aligned} $$
We put $u:=A_1+A_2+A_3$. By construction, $\Ad u\circ\varphi(P_1) = P_{21}$ and also 
$\Ad u\circ\varphi(P_{21}) = P_{1}$. Hence $\Ad u\circ\varphi(P_{22})=P_{22}$ as well. 

We have $P_{22}C^*(S_1,S_2)P_{22} \cong \OO_2=C^*(R_1,R_2)$, under the identification of 
$S_{221}S_{22}^*$ with $R_1$ and $S_{222}S_{22}^*$ with $R_2$. This isomorphism yields a 
conjugation between the restriction of $\Ad u\circ\varphi$ to $P_{22}C^*(S_1,S_2)P_{22}$ and the 
shift $\varphi_R$. Consequently, 
$$ \lambda_u(\F_2)'\cap P_{22}C^*(S_1,S_2)P_{22} = \bigcap_{k=1}^\infty (\Ad u\circ\varphi)^{k} 
   (P_{22}C^*(S_1,S_2)P_{22}) = \C P_{22}. $$

We have $P_1C^*(S_1,S_2)P_1 \cong \OO_2 = C^*(T_1,T_2)$, under the identification of 
$S_{11}S_1^*$ with $T_1$ and $S_{12}S_1^*$ with $T_2$. This isomorphism carries the restriction of 
$(\Ad u\circ\varphi)^2$ to $P_{1}C^*(S_1,S_2)P_{1}$ to the endomorphism of $C^*(T_1,T_2)$ given 
as composition $\varphi_T\circ\psi_T$, where $\psi_T$ is an endomorphism of $C^*(T_1,T_2)$ such that 
$$ \psi_T(x) = T_1xT_1^* + T_2(\Ad F_T(x))T_2^*, $$ 
where $F_T:=T_2T_1^* + T_1T_2^*$. By Lemma \ref{phipsi}, we have 
$$  \lambda_u(\F_2)'\cap P_{1}C^*(S_1,S_2)P_{1} \subseteq \bigcap_{k=1}^\infty (\Ad u\circ\varphi)^{2k} 
   (P_{1}C^*(S_1,S_2)P_{1}) = \C P_{1}. $$

We have $P_{21}C^*(S_1,S_2)P_{21} \cong \OO_2 = C^*(V_1,V_2)$, under the identification of 
$S_{211}S_{21}^*$ with $V_1$ and $S_{212}S_{21}^*$ with $V_2$. This isomorphism carries the restriction of 
$(\Ad u\circ\varphi)^2$ to $P_{21}C^*(S_1,S_2)P_{21}$ to $\psi_V\circ\varphi_V$. An argument similar 
to that from Lemma \ref{phipsi} shows that $\lambda_u(\F_2)'\cap P_{21}C^*(S_1,S_2)P_{21} = \C P_{21}$. 
Alternatively, this also easily follows from the preceding argument and the fact that 
$\Ad u\circ\varphi(P_{21})=P_1$. 

In view of the above, either $\lambda_u(\F_2)'\cap\OO_2 = \span\{P_1,P_{21},P_{22}\}\cong\C^3$, 
or $P_1$ and $P_{21}$ are equivalent in $\lambda_u(\F_2)'\cap\OO_2$. In the latter case, 
$\lambda_u(\F_2)'\cap\OO_2$ contains a subalgebra isomorphic to $M_2(\C)$ which is invariant under 
$\Ad u\circ\varphi$ and has $P_1$ and $P_{21}$ as its minimal projections.  
Suppose for a moment that this is the case. 
Then $\Ad u\circ\varphi$ restricts to a non-trivial automorphism of $M_2(\C)$, by necessity inner. 
The implementing unitary matrix $g$ is fixed by $\Ad u\circ\varphi$ and thus belongs to $\lambda_u(\OO_2)'
\cap\OO_2$. Matrix $g$ has both diagonal entries equal to $0$. Multiplying $g$ by a suitable scalar of 
modulus $1$, we can find such $g$ that is self-adjoint. Now we see that there is a unitary element 
$x$ of $\OO_2$ such that 
$$ g = S_{21}x^*S_1^* + S_1xS_{21}^* \in \lambda_u(\OO_2)'\cap\OO_2. $$ 
Now, writing $F:=S_1S_2^* + S_2S_1^*$, we compute 
$$ \begin{aligned}
{\rm Ad}u\circ\varphi(g)
&=u
(S_{11}xS_{121}^{*}+S_{121}x^{*}S_{11}^{*}
+
S_{21}xS_{221}^{*}+S_{221}x^{*}S_{21}^{*}
)
u^{*}\\
&=
S_{212}FxS_{12}^{*}+S_{12}x^{*}FS_{212}^{*}
+
S_{211}xS_{11}^{*}+S_{11}x^{*}S_{211}^{*}, 
\end{aligned} $$
and hence we get 
$$
S_{1}xS_{21}^{*}+S_{21}x^{*}S_{1}^{*}
=S_{212}FxS_{12}^{*}+S_{12}x^{*}F^{*}S_{212}^{*}
+
S_{211}xS_{11}^{*}+S_{11}x^{*}S_{211}^{*}.
$$
Multiplying by $S_{1}^{*}$ from the left-side and by 
$S_{21}$ from the right-side, we obtain 
\begin{equation}\label{technical}
x=S_{2}x^{*}FS_{2}^{*}+S_{1}x^{*}S_{1}^{*}. 
\end{equation}
Equation (\ref{technical}) implies $xS_1=S_1x^*$ and $S_1^*x=x^*S_1^*$. 
These two combined then yield $(x+x^*)S_1=S_1(x+x^*)$ and $(x-x^*)S_1 = -S_1(x-x^*)$. 
By \cite[Proposition 4]{MaTo}, both $x+x^*$ and $x-x^*$ are scalars, and thus so is $x$. 
This however contradicts (\ref{technical}). 

Thus $\lambda_u(\F_2)'\cap\OO_2 = \span\{P_1,P_{21},P_{22}\}$ and since $\tau(P_1)\neq \tau(P_{21})$, 
we conclude from Corollary 
\ref{maincriterion} that there are no unitaries $w\in\lambda_{u}(\mathcal{F}_{2})'\cap\mathcal{O}_{2}$ and 
$v\in \mathcal{F}_{2}$ such that $u=wv$. 
\hfill$\Box$
}\end{ex}

\begin{lem}\label{phipsi}
Let $\psi_T$ be an endomorphism of $C^*(T_1,T_2)\cong \OO_2$ such that 
$$ \psi_T(x) = T_1xT_1^* + T_2(\Ad F_T(x))T_2^*, $$ 
where $F_T:=T_2T_1^* + T_1T_2^*$. Then we have
$$ \bigcap_{k=1}^\infty (\varphi_T\psi_T)^k(C^*(T_1,T_2))=\C1. $$
\end{lem}
\begin{proof}
We note that 
$$ \varphi_T\psi_T(x) = T_{11}xT_{11}^* + T_{21}xT_{21}^* + T_{12}(\Ad F_T(x))T_{12}^* 
+ T_{22}(\Ad F_T(x))T_{22}^*. $$
Also, we clearly have $F_TT_1=T_2$ and $F_TT_2=T_1$. Thus $(\varphi_T\psi_T)^k(x)$ may be 
written as a finite sum of elements of the form $T_\mu XT_\mu^*$ with $|\mu|=2k$. This gives 
$$ \bigcap_{k=1}^\infty(\varphi_T\psi_T)^k(C^*(T_1,T_2))\subseteq \D_T'\cap C^*(T_1,T_2)=\D_T. $$
For a positive ineger $k$, let 
$$ Q_k := \sum_{|\mu|=k-1}T_{\mu1}T_{\mu1}^*. $$
Then a straightforward induction on $k$ shows that 
\begin{equation}\label{qeq}
Q_{2k}(\varphi_T\psi_T)^k(x) = Q_{2k}\varphi_T^{2k}(x) 
\end{equation}
for all $x\in C^*(T_1,T_2)$. Take a $d=d^*\in\D_T$ that belongs to $\bigcap_{k=1}^\infty(\varphi_T\psi_T)^k (C^*(T_1,T_2))$. Suppose $d$ is not a constant multiple of $1$. Then there exist $k\in\N$, $t\in\R$, $\epsilon>0$ and $\mu,\nu\in W_2^{2k-1}$ such that 
$$ T_{\mu1}T_{\mu1}^* d \geq (t+\epsilon)T_{\mu1}T_{\mu1}^* \;\;\; \text{and} \;\;\; 
T_{\nu1}T_{\nu1}^* d \leq (t-\epsilon)T_{\mu1}T_{\mu1}^*. $$
Let $x=x^*\in \D_2$ be such that $d=(\varphi_T\psi_T)^k(x)$. Then $Q_{2k}d=Q_{2k}\varphi_T^{2k}(x)$. 
Since $T_{\mu1}T_{\mu1}^*\leq Q_{2k}$ and $T_{\nu1}T_{\nu1}^*\leq Q_{2k}$, we get 
$$ \begin{aligned}
T_{\mu1}xT_{\mu1}^* & = T_{\mu1}T_{\mu1}^*Q_{2k}\varphi_T^{2k}(x)\geq(t+\epsilon)   T_{\mu1}T_{\mu1}^*, \;\;\; \text{ and } \\
T_{\nu1}xT_{\nu1}^* & = T_{\nu1}T_{\nu1}^*Q_{2k}\varphi_T^{2k}(x)\leq(t-\epsilon)   
T_{\nu1}T_{\nu1}^*. 
\end{aligned} $$
This, however, is a contradiction. Indeed, since $T_{\mu 1}$ and $T_{\nu 1}$ are isometries, the above two 
inequalities would imply that both $x \geq (t+\epsilon)$ and $x \leq (t-\epsilon)$. Consequently, 
$$ \bigcap_{k=1}^\infty (\varphi_T\psi_T)^k(C^*(T_1,T_2))=\C1, $$
as required. 
\end{proof}


\end{document}